\documentclass[11pt,a4paper]{article}
\usepackage{amssymb}
\usepackage{mathrsfs}
\usepackage{amsmath}
\usepackage{amsfonts, amsthm, amssymb}
\usepackage{amsfonts}
\usepackage{color}
\usepackage{float}
\usepackage{graphicx}
\usepackage{subfigure}
\usepackage{caption}

\usepackage{amsmath}
\numberwithin{equation}{section}
\newtheorem{lem}{Lemma}[section]
\newtheorem{thm}[lem]{Theorem}
\newtheorem{cor}[lem]{Corollary}
\newtheorem{claim}{Claim}[section]

\newtheorem*{rem}{Remark}

\begin{document}

\title{Triangle-free Graphs with Large Minimum Common Degree}
\author{Jian Wang\footnote{Department of Mathematics, Taiyuan University of Technology, Taiyuan 030024, P. R. China. E-mail:wangjian01@tyut.edu.cn.}\quad\quad
Weihua Yang\footnote{Department of Mathematics, Taiyuan University of Technology, Taiyuan 030024, P. R. China. E-mail:yangweihua@tyut.edu.cn. Research supported by NSFC No. 12371356. }\quad\quad
Fan Zhao\footnote{Department of Mathematics, Taiyuan University of Technology, Taiyuan 030024, P. R. China. E-mail:zhaofan19990804@163.com.}
\\
}
\date{}
\maketitle

\begin{abstract}
Let $G$ be a graph. For $x\in V(G)$, let $N(x)=\{y\in V(G)\colon xy\in E(G)\}$.  The {\it minimum common degree} of $G$, denoted by $\delta_{2}(G)$, is defined as the minimum  of $|N(x)\cap N(y)|$ over all non-edges $xy$ of $G$. In 1982, H\"{a}ggkvist showed that every triangle-free graph with minimum degree greater than $\lfloor\frac{3n}{8}\rfloor$ is homomorphic to a cycle of length 5. In this paper, we prove that every  triangle-free graph with minimum common degree greater than $\lfloor\frac{n}{8}\rfloor$ is homomorphic to a cycle of length 5, which implies H\"{a}ggkvist's result. The balanced blow-up of the M\"{o}bius ladder graph shows that it is best possible.
\end{abstract}

\noindent{\bf Keywords:} triangle-free; minimum common degree;  the blow-up of $C_5$.

\section{Introduction}

Let $H$ be a fixed graph. We say that a graph $G$ is {\it $H$-free} if it does not contain $H$ as a subgraph. The {\it Tur\'{a}n number} ${\rm ex}(n,H)$ is defined as the maximum number of edges in an $H$-free graph on $n$ vertices. Let $T_r(n)$ be a complete $r$-partite graph with partite classes of sizes $\lfloor\frac{n}{r}\rfloor$ or $\lceil\frac{n}{r}\rceil$.

In 1907, Mantel determined ${\rm ex}(n,K_3)$, which is the starting point of the extremal graph theory.

\begin{thm}[\cite{mantel}]
\[
{\rm ex}(n,K_3) = \left\lfloor\frac{n^2}{4}\right\rfloor.
\]
Moreover, $T_2(n)$ is the unique triangle-free graph attaining the maximum number of edges.
\end{thm}

In 1943, Tur\'{a}n \cite{turan} proved that ${\rm ex}(n,K_{r+1}) =e(T_r(n))$. Since then the study of Tur\'{a}n number became a central topic in extremal graph theory. We refer to \cite{sidorenko,keevash} for surveys on Tur\'{a}n problems for graphs and hypergraphs.

Let $\delta(G)$ denote the minimum degree of $G$. In 1974, Andr\'{a}sfai, Erd\"{o}s and  S\'{o}s \cite{aes} proved the following result.

\begin{thm}[\cite{aes}]\label{AES}
If $G$ is a triangle-free graph on $n$ vertices with $\delta(G)>\frac{2n}{5}$, then $G$ is bipartite.
\end{thm}

Note that Theorem \ref{AES} implies Mantel's theorem. Indeed, if $G$ is bipartite then $e(G)\leq\lfloor\frac{n^2}{4}\rfloor$ follows. Otherwise there is a vertex with degree at most $\frac{2n}{5}$ in $G$ and then Mantel's theorem follows from an induction argument on $n$.

Let $F$ be a graph on $k$ vertices. An {\it $F$-blow-up} is a graph obtained from $F$ by replacing each vertex with an independent set and replacing each edge by a complete bipartite graph. We say an $F$-blow-up on $n$ vertices is {\it balanced} if each vertex of $F$  is replaced by an independent set of size $\lfloor n/k\rfloor$ or $\lceil n/k\rceil$. A balanced $C_5$-blow-up on $n$ vertices shows that Theorem \ref{AES} is sharp (as shown in Figure 1 (a)).

\begin{figure}[thbp!]
	\centering
	\begin{minipage}[t]{0.4\linewidth}
		\centering
		\includegraphics[width=0.5\linewidth]{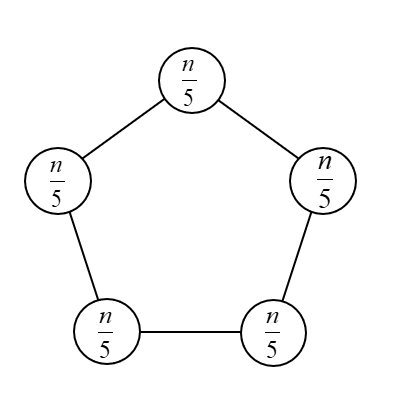}\\(a)\\
	\end{minipage}
	\begin{minipage}[t]{0.4\linewidth}
		\centering
		\includegraphics[width=0.5\linewidth]{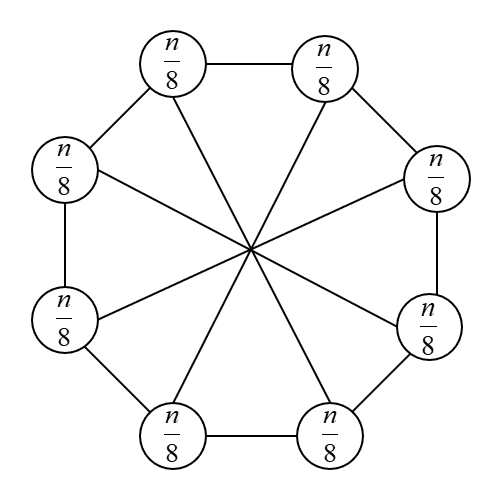}\\(b)\\
	\end{minipage}
	\caption{(a) A balanced $C_5$-blow-up (b) A balanced $H_M$-blow-up}
\end{figure}

Let $G, H$ be graphs. We say that {\it $G$ is homomorphic to $H$} if there is a map $\sigma$ from $V(G)$ to $V(H)$ such that $\sigma(u)\sigma(v)\in E(H)$ for every $uv\in E(G)$. Note that if $G$ is a subgraph of an $H$-blow-up then $G$ is homomorphic to $H$.

In 1982, H\"{a}ggkvist \cite{hagg} extended the Andr\'{a}sfai-Erd\H{o}s-S\'{o}s Theorem to the following form.

\begin{thm}[\cite{hagg}]\label{hagg}
If $G$ is a triangle-free graph on $n$ vertices with $\delta(G)>\lfloor\frac{3n}{8}\rfloor$, then $G$ is homomorphic to $C_5$.
\end{thm}

Define the  {\it M\"{o}bius ladder} $H_M$ to be the graph obtained from a cycle of length 8 by adding 4 chords joining vertices of distance $4$ on the cycle. A balanced $H_M$-blow-up on $n$ vertices shows that Theorem \ref{hagg} is sharp (as shown in Figure 1 (b)).

The {\it minimum common degree} of $G$, denoted by $\delta_{2}(G)$, is defined as the minimum of $|N(x)\cap N(y)|$ over all non-edges $xy$ of $G$. We say that $G$ is {\it maximal triangle-free} if $G$ is triangle-free and any addition of an extra edge to $G$ creates a triangle. Recently, Liu, Shangguan, Skokan and Xu \cite{LSSX} proved that every maximal triangle-free graph with $\delta_{2}(G)\geq\varepsilon n$ is a blow-up of some triangle-free graph with at most $2^{O(\frac{3}{\varepsilon}\log\frac{1}{\varepsilon})}$ vertices.

Motivated by  H\"{a}ggkvist's result and Liu, Shangguan, Skokan and Xu's result,  we consider triangle-free graphs with large minimum common degree.  Our main result is the following.

\begin{thm}\label{main}
Let $G$ be a triangle-free graph on $n$ vertices. Then (i) and (ii) hold.
\begin{itemize}
  \item[(i)]If $\delta_{2}(G)>\lfloor\frac{n}{5}\rfloor$ and $n\geq 5$, then $G$ is bipartite.
  \item[(ii)]If  $\delta_{2}(G)>\lfloor\frac{n}{8}\rfloor$ and $n\geq 8$, then $G$ is homomorphic to $C_5$.
\end{itemize}
\end{thm}

We need the following lemma proved in \cite{LSSX}. For self-containedness,  we include a proof of it.

\begin{lem}[\cite{LSSX}]\label{lem-1}
If $G$ is a maximal triangle-free graph on $n$ vertices with $\delta(G)>(\frac{1}{3}+\alpha)n$, then $\delta_{2}(G)>3\alpha n$.
\end{lem}

\noindent\emph{\textbf{Proof}.} Let $G$ be a maximal triangle-free graph on $n$ vertices with $\delta(G)>(\frac{1}{3}+\alpha)n$. Take $x,y\in V(G)$ arbitrarily such that $xy\notin E(G)$. Note that  $|N(x)|>(\frac{1}{3}+\alpha)n$ and $|N(y)|>(\frac{1}{3}+\alpha)n$. We claim that $N(x)\cap N(y)\neq\emptyset$. Indeed, otherwise $G+xy$ is also triangle-free, contradicting the maximality of $G$.

Let $w\in N(x)\cap N(y)$. Since $G$ is triangle-free, $w$ has no neighbor in $N(x)\cup N(y)$. Thus,
\begin{align}
n=|V(G)|&\geq |N(x)\cup N(y)| +|N(w)| \nonumber\\
&=|N(x)|+|N(y)|-|N(x)\cap N(y)|+|N(w)|\nonumber\\
&>3\cdot\left(\frac{1}{3}+\alpha\right)n-|N(x)\cap N(y)|.\nonumber
\end{align}
It follows that $|N(x)\cap N(y)|>3\alpha n$.
$\hfill\square$

Applying Lemma \ref{lem-1} with $\alpha=\frac{1}{15}$ and $\frac{1}{24}$,  we see that $\delta(G)>2n/5$ implies $\delta_2(G)>n/5$ and  $\delta(G)>3n/8$ implies $\delta_2(G)>n/8$. Thus, Theorem \ref{AES} follows from Theorem \ref{main} (i) and   Theorem \ref{hagg} follows from Theorem \ref{main} (ii).

Let $G_1$ be a $C_4$-blow-up with vertices of $C_4$ being replaced by independent sets of sizes $n/8$, $3n/8$, $n/8$ and $3n/8$ consecutively. Let $G_2$ be a $C_5$-blow-up with vertices of $C_5$ being replaced by independent sets of sizes $n/7$, $2n/7$, $n/7$, $n/7$ and $2n/7$ consecutively. It is easy to check that $\delta_2(G_1)=n/4>n/5$ and $\delta_2(G_2)=n/7>n/8$ but $\delta(G_1)=n/4<2n/5$ and $\delta(G_2)=2n/7<3n/8$ (as shown in Figure 2). Thus the condition $\delta_2(G)> 3\alpha n$ is strictly stronger than $\delta(G)>(\frac{1}{3}+\alpha)n$.

\begin{figure}[thbp!]
	\centering
	\begin{minipage}[t]{0.4\linewidth}
		\centering
		\includegraphics[width=0.6\linewidth]{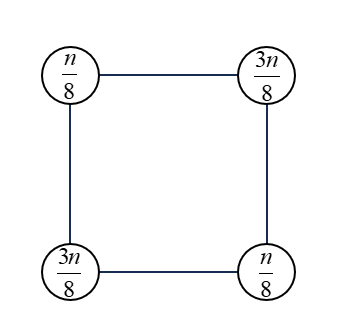}\\
		$G_1$\\
	\end{minipage}
	\begin{minipage}[t]{0.4\linewidth}
		\centering
		\includegraphics[width=0.6\linewidth]{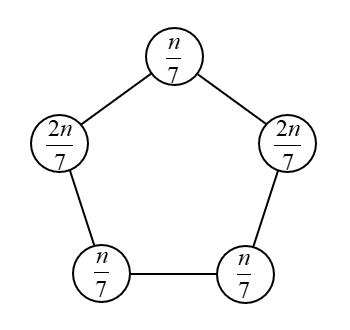}\\
		$G_2$\\
	\end{minipage}
	\caption{The graphs $G_1$ and $G_2$.}
\end{figure}

For $C_5$-free graphs, we prove the following result.

\begin{thm}\label{C5}
If $G$ is a $C_{5}$-free graph on $n$ vertices $(n\geq 5)$ with $\delta_{2}(G)\geq 3$, then $G$ is bipartite.
\end{thm}

\section{Some useful lemmas}
In this section, we prove some lemmas that are needed in the proofs.

\begin{lem}\label{lem-2}
If $G$ is a graph on $n$ vertices with $\delta_{2}(G)\geq 1$, then either $G$ is bipartite or the shortest odd cycle of $G$ has length  3 or 5.
\end{lem}

\noindent\emph{\textbf{Proof}.}  Suppose that $G$ is non-bipartite. Let $C=v_1v_2\ldots v_\ell v_1$ be a shortest odd cycle of $G$. Note that $C$ is an induced cycle. If $\ell\geq 7$, then $N(v_{1})\cap N(v_{4})\neq\emptyset$. It follows that there is some $v\in N(v_{1})\cap N(v_{4})$. Since $C$ is chord-free, $v\in V(G)\setminus V(C)$. Then $vv_{1}v_{2}v_{3}v_{4}v$ is a $C_5$, contradicting the minimality of $\ell$. Thus the lemma is proven.
$\hfill\square$

\begin{cor}
If $G$ is a $\left\lbrace C_{3},C_{5}\right\rbrace$-free graph on $n$ vertices with $\delta_{2}(G)\geq 1$, then $G$ is bipartite.
\end{cor}

 We need the following simple but useful lemma.

\begin{lem}\label{lem-3}
Suppose  that $G$ is a triangle-free graph on $n$ vertices and $x_{1}x_{2},y_{1}y_{2}\notin E(G)$. Let $X=N(x_{1})\cap N(x_{2})$, $Y=N(y_{1})\cap N(y_{2})$. If $x_iy_j\in E(G)$ for some $i\in \{1,2\}$, and $j\in \{1,2\}$, then $X,Y$ are disjoint.
\end{lem}

\noindent\emph{\textbf{Proof}.} Without loss of generality,  assume  $x_1y_1\in E(G)$. If there is some $w\in X\cap Y$, then $x_1w$, $wy_1\in E(G)$.  It follows that $wx_{1}y_{1}w$ is a triangle, a contradiction.
$\hfill\square$

\begin{cor}\label{cor-2.5}
Suppose that  $G$ is a triangle-free graph that contains a $C_5$,  $v_1v_2\ldots v_5v_1$. Let $D_i=N(v_{i-1})\cap N(v_{i+1})$, $i=1,2,\ldots,5$. Then $D_1,D_2,\ldots, D_5$ are pairwise disjoint.
\end{cor}
\begin{proof}
Since $v_{i-1}v_{i}\in E(G)$, by Lemma \ref{lem-3} we infer that $D_{i-1}\cap D_{i}\neq \emptyset$ and $D_{i-2}\cap D_{i+1}\neq \emptyset$. Thus $D_1,D_2,\ldots,D_5$ are pairwise disjoint.
\end{proof}

\begin{lem}\label{lem-4}
If $G$ is a triangle-free graph that contains $H_M$ as a subgraph, then $\delta_{2}(G)\leq\lfloor\frac{n}{8}\rfloor$.
\end{lem}

\noindent\emph{\textbf{Proof}.} Assume that a cycle $v_{1}v_{2}v_{3}v_{4}v_{5}v_{6}v_{7}v_{8}v_{1}$ with chords $v_{1}v_{5},v_{2}v_{6},v_{3}v_{7}$ and $v_{4}v_{8}$ is  a copy of  $H_M$  in $G$. Since $G$ is triangle-free, $v_iv_{i+2}\notin E(G)$ for $i=1,2,\ldots,8$.  Let $S_i=N(v_{i})\cap N(v_{i+2})$ with subscripts modulo 8. Since $v_{i-1}v_{i}\in E(G)$, by Lemma \ref{lem-3} we infer that $S_{i-1}\cap S_{i}\neq \emptyset$ and $S_{i-3}\cap S_{i}\neq \emptyset$. Since $v_iv_{i+4}\in E(G)$,  by Lemma \ref{lem-3} we infer that $S_i\cap S_{i+2}\neq \emptyset$ and $S_i\cap S_{i+4}\neq \emptyset$. Thus $S_1,S_2,\ldots,S_8$ are pairwise disjoint. If $\delta_{2}(G)>\lfloor\frac{n}{8}\rfloor$  then $|S_i|>\lfloor\frac{n}{8}\rfloor$ for all $i=1,2,\ldots,8$. It follows that
\[
n=|V(G)|=\sum_{1\leq i\leq 8} |S_i| \geq 8\left(\left\lfloor\frac{n}{8}\right\rfloor+1\right)>n,
\]
a contradiction. Thus $\delta_{2}(G)\leq\lfloor\frac{n}{8}\rfloor$.
$\hfill\square$

\section{Proof of  Theorems \ref{main} and \ref{C5}}
In this section, we prove Theorems \ref{main} and \ref{C5}.

\noindent\emph{\textbf{Proof of Theorem \ref{main} (i)}.} Let $G$ be a triangle-free graph on $n$ vertices with $\delta_{2}(G)>\lfloor\frac{n}{5}\rfloor$. By Lemma \ref{lem-2},  either $G$ is bipartite or $G$ contains a $C_5$. In the former case, there is nothing to prove. Thus we assume that $G$ contains a $C_5$ and let $C=v_1v_2v_3v_4v_5v_1$ be such a $C_5$.  Let $D_i=N(v_{i-1})\cap N(v_{i+1})$ with subscripts modulo 5, $i=1,2,\ldots,5$.
Since $v_{i-1}v_{i+1}\notin E(G)$ and $\delta_2(G)> \lfloor\frac{n}{5}\rfloor$, we have $|D_i|\geq\lfloor\frac{n}{5}\rfloor+1$. By Corollary \ref{cor-2.5}, $D_1,D_2,\ldots,D_5$ are pairwise disjoint. It follows that
\[
n=|V(G)|=\sum_{1\leq i\leq 5} |D_i| \geq 5\left(\left\lfloor\frac{n}{5}\right\rfloor+1\right)>n,
\]
a contradiction. Thus $G$ is bipartite.
$\hfill\square$

\begin{rem}
If $G$ is a triangle-free graph on $n$ vertices and $\delta_{2}(G)=\frac{n}{5}$( $n$ is divisible by 5), then
\[
n=|V(G)|=\sum_{1\leq i\leq 5} |D_i| \geq 5\times\frac{n}{5}=n.
\]
It follows that $|D_i|=\frac{n}{5}$, $i=1,2,\ldots,5$. That is, $(D_1,D_2,\ldots,D_5)$ forms a balanced $C_{5}$-blow-up.
\end{rem}

For $D\subset V(G)$, we use $G[D]$ to denote the  subgraph of $G$ induced by $D$. 

\noindent\emph{\textbf{Proof of Theorem \ref{main} (ii)}.} Let $G$ be a triangle-free graph on $n$ vertices with $\delta_{2}(G)>\lfloor\frac{n}{8}\rfloor$. We may assume that $G$ is not bipartite. Then by Lemma \ref{lem-2}, $G$  contains a $C_5$. Let $C=v_1v_2v_3v_4v_5v_1$ be such a $C_5$. Note that $C$ is an induced cycle since $G$ is triangle-free, i.e., $v_{i-1}v_{i+1}\notin E(G)$.   Let
$D_{i}=N(v_{i-1})\cap N(v_{i+1})$ with subscripts modulo 5, $i=1,2,3,4,5$ and let $D=\cup_{1\leq i\leq 5} D_i$. By Corollary \ref{cor-2.5}, $D_1,D_2,\ldots,D_5$ are pairwise disjoint.

\begin{claim}
$G[D]$ is homomorphic to $C_{5}$.
\end{claim}

\noindent\emph{\textbf{Proof}.} Since $\delta_2(G)>\lfloor\frac{n}{8}\rfloor$, $|D_i|>\lfloor\frac{n}{8}\rfloor$.   Since $G$ is triangle-free, each $D_i$ is an independent set. We claim that each edge in $E(G[D])$ is between $D_i$ and $D_{i+1}$. Indeed, otherwise by symmetry assume that there is some $xy\in E(G[D])$ with $x\in D_1$ and $y\in D_3$. Then $xv_2yx$ is a triangle, a contradiction. Thus $G[D]$ is homomorphic to $C_{5}$.
$\hfill\square$

Clearly $|D|=\sum_{1\leq i\leq 5}|D_i|>5\lfloor\frac{n}{8}\rfloor$.  Let $W=V(G)\setminus D$. If $W=\emptyset$ then we are done. Thus we assume $W\neq \emptyset$.

\begin{claim}\label{claim-1}
For each $x\in W$, $x$ has neighbors in   at least two $D_{i}$'s.
\end{claim}

\noindent\emph{\textbf{Proof}.} If $x$ has neighbors in at most one $D_{i}$, then by symmetry assume that $x$  has no neighbor in $D_2\cup D_3\cup D_4\cup D_5$. As $v_i\in D_i$, we have $xv_i\notin E(G)$.
 Let $W_i=N(x)\cap N(v_i)$, $i=2,3,4,5$. Then $\delta_2(G)>\lfloor\frac{n}{8}\rfloor$ implies $|W_i|>\lfloor\frac{n}{8}\rfloor$.

 We claim that $D_i\cap W_j=\emptyset$ for all  $i\in \{2,3,4,5\}$ and $j\in \{2,3,4\}$. Indeed, if $D_i\cap W_j\neq\emptyset$ for some  $i\in \{2,3,4,5\}$ and $j\in \{2,3,4\}$, then $N(x)\cap D_i\neq\emptyset$, contradicting our assumption that $x$  has no neighbor in $D_2\cup D_3\cup D_4\cup D_5$. Thus $D_i\cap W_j=\emptyset$ for all  $i\in \{2,3,4,5\}$ and $j\in \{2,3,4\}$.

Since $v_2v_3\in E(G)$, by Lemma \ref{lem-3} we have $W_2\cap W_3=\emptyset$. Since $v_3v_4\in E(G)$, by Lemma \ref{lem-3} we have $W_3\cap W_4=\emptyset$. If $z \in W_2\cap W_4$ then $z\in D_3$, then $x$ has neighbor $z$ in $D_3$, contradicting our assumption that $x$  has no neighbor in $D_2\cup D_3\cup D_4\cup D_5$. Thus we also have   $W_2\cap W_4=\emptyset$. Thus $D_2,D_3,D_4,D_5,W_2,W_3,W_4$ are pairwise disjoint.

 \begin{figure}[H]
	\centering
	\begin{tabular}{@{\extracolsep{\fill}}c@{}c@{\extracolsep{\fill}}}
		\includegraphics[width=0.3\linewidth]{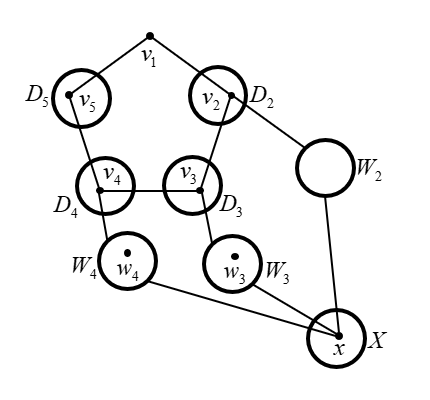}
	\end{tabular}
	\caption{$G$ in Claim \ref{claim-1}}
	\label{fig:image_with_table}
\end{figure}

Let us take $w_3\in W_3$ and $w_4\in W_4$. Then $w_3w_4\notin E(G)$ since $w_3,w_4\in N(x)$ and $G$ is triangle-free. Let $X=N(w_3)\cap N(w_4)$. Then $|X|>\lfloor\frac{n}{8}\rfloor$. We want to show that  $X,D_2,D_3,D_4,D_5,W_2,W_3,W_4$ are pairwise disjoint. Since $w_3v_3\in E(G)$, by Lemma \ref{lem-3} we know $X\cap D_2=\emptyset=X\cap D_4$. By $w_4v_4\in E(G)$ and Lemma \ref{lem-3}, $X\cap D_3=\emptyset = X\cap D_5$.  Similarly, using $xw_3\in E(G)$, $X\cap W_j=\emptyset$ follows from Lemma \ref{lem-3} for $j=2,3,4$.  Thus, $X,D_2,D_3,D_4,D_5,W_2,W_3,W_4$ are pairwise disjoint (as shown in Figure 3). Then
\[
n=|V(G)|\geq 8\cdot\left(\left\lfloor\frac{n}{8}\right\rfloor+1\right)>n,
\]
a contradiction. Therefore $x$ has neighbors in   at least two $D_{i}$'s.
$\hfill\square$

\begin{claim}\label{claim-2}
For each $x\in W$, $x$ has no neighbor in two consecutive $D_{i}'s$.
\end{claim}

\noindent\emph{\textbf{Proof}.} Suppose for contradiction that  $x$ has neighbors in $D_{1}$ and $D_{2}$. Clearly $xv_1,xv_2\notin E(G)$ since $G$ is triangle-free. Let $x_1\in N(x)\cap D_{1}$ and $x_2\in N(x)\cap D_{2}$.

 If there exists $x_4\in N(x)\cap D_4$,  then the graph with the vertex set $\left\lbrace v_1,v_2,v_3,x_4,v_5,x,x_1,x_2\right\rbrace$ and the edge set
\[\left\lbrace v_1v_2,v_2v_3,v_3x_4,x_4v_5,v_5v_1,xx_1,xx_2,xx_4,x_1v_5,x_1v_2,x_2v_1,x_2v_3\right\rbrace\]
is a copy of $H_M$ in  $G$. By Lemma \ref{lem-4} we obtain  $\delta_{2}(G)\leq\lfloor\frac{n}{8}\rfloor$, a contradiction. Thus we may assume that $N(x)\cap D_4=\emptyset$.

Let $W_i=N(x)\cap N(v_i)$, $i=3,5$. Since $x_1\in N(x)\cap D_1$ and $x_2\in N(x)\cap D_2$, we have $x_1v_5,x_2v_3\in E(G)$. It follows that $x_1\in W_5$ and  $x_2\in W_3$. By $\delta_2(G)>\lfloor\frac{n}{8}\rfloor$ we have  $|W_i|>\lfloor\frac{n}{8}\rfloor$, $i=3,5$. Since $v_3v_4\in E(G)$,    by Lemma \ref{lem-3} we know $W_3\cap D_3=\emptyset = W_3\cap D_5$. Since $v_4v_5\in E(G)$  by Lemma \ref{lem-3} we have $W_5\cap D_3=\emptyset = W_5\cap D_5$. Since $N(x)\cap D_4=\emptyset$ and $W_3\cup W_5\subset N(x)$, we also have $W_3\cap D_4=\emptyset=W_5\cap D_4$. Thus $W_3\cup W_5$ is disjoint to $D_3\cup D_4\cup D_5$. If $z\in W_3\cap W_5$ then $z\in D_4$, contradicting $N(x)\cap D_4=\emptyset$. Thus $W_3,W_5,D_3, D_4, D_5$ are pairwise disjoint.

\begin{figure}[H]
	\centering
	\begin{tabular}{@{\extracolsep{\fill}}c@{}c@{\extracolsep{\fill}}}
		\includegraphics[width=0.3\linewidth]{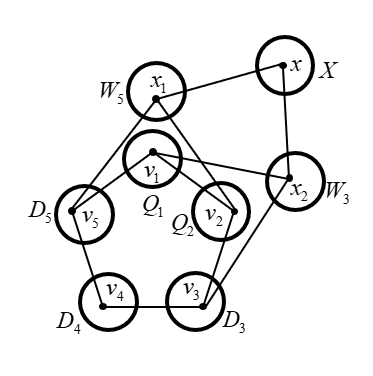}
	\end{tabular}
	\caption{$G$ in  Claim \ref{claim-2} }
	\label{fig:image_with_table}
\end{figure}

 Since $x_1,x_2\in N(x)$, we have $x_1x_2\notin E(G)$. Let $X=N(x_{1})\cap N(x_{2})$.  Then $|X|>\lfloor\frac{n}{8}\rfloor$. Using $x_1v_2\in E(G)$, by Lemma \ref{lem-3} we have $X\cap D_3=\emptyset$.  Using $x_2v_3\in E(G)$, by Lemma \ref{lem-3} we have $X\cap D_4=\emptyset$. Using $x_2v_1\in E(G)$, by Lemma \ref{lem-3} we have $X\cap D_5=\emptyset$. Similarly, $X\cap W_3=\emptyset =X\cap W_5$ since $xx_1\in E(G)$. Thus $X,W_3,W_5,D_3, D_4, D_5$ are pairwise disjoint.

 Note that $x_2v_5$, $x_1v_3\notin E(G)$. We can define $Q_{1}=N(x_2)\cap N(v_{5})$ and  $Q_{2}=N(x_1)\cap N(v_{3})$ with $|Q_{1}|>\lfloor\frac{n}{8}\rfloor$, $|Q_{2}|>\lfloor\frac{n}{8}\rfloor$. By Lemma \ref{lem-3}, using $v_4v_5\in E(G)$ we have $Q_1\cap D_3=\emptyset=Q_1\cap D_5$, using $x_2v_3\in E(G)$ we have $Q_1\cap D_4=\emptyset$, using $xx_2\in E(G)$ we have $Q_1\cap W_3=\emptyset=Q_1\cap W_5$, using $x_1v_5$ we have $Q_1\cap X=\emptyset$. Similarly,  $Q_2\cap D_3=\emptyset=Q_2\cap D_5$ follows from $v_3v_4\in E(G)$, $Q_2\cap D_4=\emptyset$ follows from $x_1v_5\in E(G)$, $Q_2\cap W_3=\emptyset =Q_2\cap W_5$ follows from $xx_1\in E(G)$ and  $Q_2\cap X=\emptyset=Q_2\cap Q_1$ follows from $x_2v_3\in E(G)$. Thus, $X,Q_1,Q_2,W_3,W_5,D_3,D_4,D_5$ are pairwise disjoint (as shown in Figure 4). Then
\[
n=|V(G)|\geq 8\cdot\left(\left\lfloor\frac{n}{8}\right\rfloor+1\right)>n,
\]
a contradiction. Therefore $x$ has no neighbor in two consecutive $D_{i}'s$.
$\hfill\square$

For $i=1,2,3,4,5$, define
\[W_{i}=\left\lbrace x\in W\colon N(x)\cap D_{i-1}\neq\emptyset,\ N(x)\cap D_{i+1}\neq\emptyset,\ N(x)\cap (D\setminus (D_{i-1}\cup D_{i+1}))=\emptyset\right\rbrace.\]
By Claims \ref{claim-1} and \ref{claim-2}, $(W_1,W_2,\ldots,W_5)$ is a partition of $W$.

\begin{claim}\label{claim-3}
$W_{i}$ is an independent set in $G$ for each $i=1,2,3,4,5$.
\end{claim}
\noindent\emph{\textbf{Proof}.}
Suppose for contradiction that  $xy$ is an edge in $W_{1}$. Let $x_{2}\in N(x)\cap D_{2}$, $x_{5}\in N(x)\cap D_{5}$, $y_{2}\in N(y)\cap D_{2}$ and $y_{5}\in N(y)\cap D_{5}$. Since $G$ is triangle-free, $x_2,x_5,y_2,y_5$ are distinct from each other.

\begin{figure}[H]
	\centering
	\begin{tabular}{@{\extracolsep{\fill}}c@{}c@{\extracolsep{\fill}}}
		\includegraphics[width=0.4\linewidth]{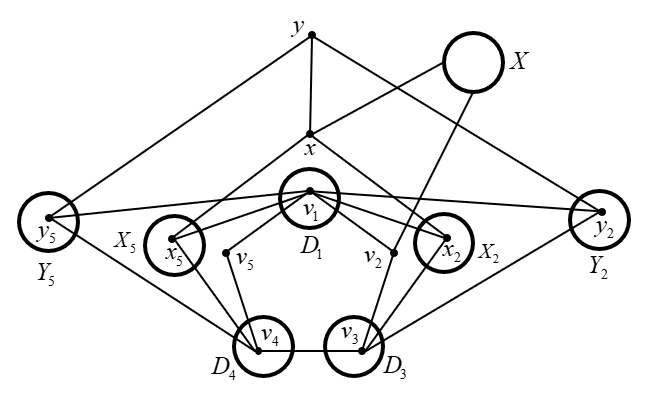}
	\end{tabular}
	\caption{$G$ in Claim \ref{claim-3}}
	\label{fig:image_with_table}
\end{figure}

Since $x,y\in W_1$, by the definition of $W_1$ there is no edge between $v_{3},v_{4}$ and $x,y$. Let $X_{2}=N(x)\cap N(v_{3})$, $X_{5}=N(x)\cap N(v_{4})$, $Y_{2}=N(y)\cap N(v_{3})$ and  $Y_{5}=N(y)\cap N(v_{4})$. Clearly $x_2\in X_2$, $y_2\in Y_2$, $x_5\in X_5$ and  $y_5\in Y_5$. By $\delta_2(G)>\lfloor\frac{n}{8}\rfloor$, we have  $|X_{2}|>\lfloor\frac{n}{8}\rfloor$, $|X_{5}|>\lfloor\frac{n}{8}\rfloor$, $|Y_{2}|>\lfloor\frac{n}{8}\rfloor$ and $|Y_{5}|>\lfloor\frac{n}{8}\rfloor$.
 
 Since $x,y\in W_1$, by the definition of $W_1$ we infer that 
 \[
 X_i \cap (D_1\cup D_3\cup D_4)=\emptyset= Y_j\cap (D_1\cup D_3\cup D_4),\ i=2,5, \ j=2,5.
 \]
 Thus $(X_2\cup X_5\cup Y_2\cup Y_5)\cap (D_1\cup D_3\cup D_4)=\emptyset$.

Applying Lemma \ref{lem-3} with $v_3v_4$, we know $X_2\cap X_5=\emptyset=Y_2\cap Y_5$. Applying Lemma \ref{lem-3} with $xy$,  $X_i\cap Y_j=\emptyset$ for $i=2,5$ and $j=2,5$. Thus, $D_1,D_3,D_4,X_2,X_5,Y_2,Y_5$ are pairwise disjoint.

Since $G$ is triangle-free,  at most one of $xv_2$ and $yv_2$ is an edge in $E(G)$. Without loss of generality,  assume $xv_2\notin E(G)$. Then let $X=N(x)\cap N(v_{2})$. Since $x\in W_1$, $x$ have no neighbor in $D_1,D_3,D_4$. It follows that $X\cap D_i=\emptyset$, $i=1,3,4$.   If there exists $z\in X\cap X_5$, then $z\in N(v_2)\cap N(v_4)$. It follows that  $z\in D_3$, contradicting $X\cap D_3=\emptyset$. Thus $X\cap X_5=\emptyset$. Applying Lemma \ref{lem-3} with $v_2v_3$, we know $X\cap X_2=\emptyset$. Applying Lemma \ref{lem-3} with $xy$,  $X\cap Y_2=\emptyset=X\cap Y_5$. Thus, $X, D_1,D_3,D_4,X_2,X_5,Y_2,Y_5$ are pairwise disjoint (as shown in Figure 5). Then
\[
n=|V(G)|\geq 8\cdot\left(\left\lfloor\frac{n}{8}\right\rfloor+1\right) >n,
\]
a contradiction.
$\hfill\square$

\begin{claim}\label{claim-4}
There is no edge between $W_{i}$ and $W_{i+2}$ for $i=1,2,3,4,5$.
\end{claim}

\noindent\emph{\textbf{Proof}.} Suppose for contradiction that $xy$ is an edge with $x\in W_1$ and $y\in W_3$.  Let $x_{2}\in N(x)\cap D_{2}$, $x_{5}\in N(x)\cap D_{5}$,  $y_{2}\in N(y)\cap D_{2}$, $y_{4}\in N(y)\cap D_{4}$. Since $x\in W_1$, $v_{3}x,v_{4}x\notin E(G)$. Let $X_{2}=N(x)\cap N(v_{3})$ and  $X_{5}=N(x)\cap N(v_{4})$. Clearly $x_2\in X_2$ and $x_5\in X_5$. Since $y\in W_3$, $v_{1}y,v_{5}y\notin E(G)$. Let $Y_{2}=N(y)\cap N(v_{1})$ and $Y_{4}=N(y)\cap N(v_{5})$. Clearly $y_2\in Y_2$ and $y_4\in Y_4$.

Since $x\in W_1$ and $y\in W_3$, by the definitions of $W_1$ and $W_3$ we infer that
 \[
 X_i \cap (D_1\cup D_3)=\emptyset=Y_j\cap (D_1\cup D_3),\ i=2,5, \ j=2,4.
 \] Thus $(X_2\cup X_5\cup Y_2\cup Y_4)\cap (D_1\cup D_3)=\emptyset$.

Applying Lemma \ref{lem-3} with $v_3v_4$ and $v_1v_5$, we know $X_2\cap X_5=\emptyset=Y_2\cap Y_4$. Applying Lemma \ref{lem-3} with $xy$,  $X_i\cap Y_j=\emptyset$ for $i=2,5$ and $j=2,4$. Thus $X_2, X_5, Y_2, Y_4,D_1, D_3$ are pairwise disjoint.

 Since $x,y \in W$, at most one of $xv_2$ and $xv_5$ is an edge of $G$ and at most one of $yv_2$ and $yv_4$ is an edge of $G$. We distinguish four cases.

 {\bf Case 1. } $xv_2,yv_2\notin E(G)$.

 \begin{figure}[H]
	\centering
	\begin{tabular}{@{\extracolsep{\fill}}c@{}c@{\extracolsep{\fill}}}
		\includegraphics[width=0.4\linewidth]{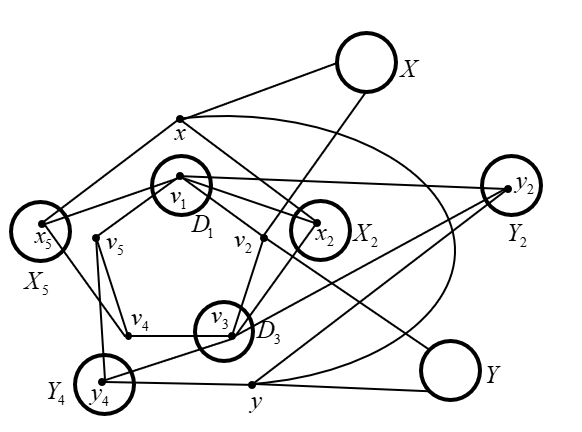}
	\end{tabular}
	\caption{$G$ in Case 1 of Claim \ref{claim-4} }
	\label{fig:image_with_table}
\end{figure}
Let $X=N(x)\cap N(v_{2})$ and $Y=N(y)\cap N(v_{2})$. Since $x\in W_1$, $X\cap D_1=\emptyset =X\cap D_3$. Since $y\in W_3$, $Y\cap D_1=\emptyset =Y\cap D_3$. Since $xy\in E(G)$, $X\cap Y=\emptyset$, $X\cap Y_j=\emptyset$ for $j=2,4$ and $X_i\cap Y=\emptyset$ for $i=2,5$. Thus $X\cap (Y\cup Y_2\cup Y_4\cup D_1\cup D_3)=\emptyset$ and $Y\cap (X\cup X_2\cup X_5\cup D_1\cup D_3)=\emptyset$.

Since $v_2v_3\in E(G)$, $X\cap X_2=\emptyset$.
If there exists $u\in X\cap X_5$ then $u\in D_3$, contradicting the fact that $N(x)\cap D_3=\emptyset$. Thus $X\cap X_5=\emptyset$. Since $v_1v_2\in E(G)$, $Y\cap Y_2=\emptyset$. If there exists $u\in Y\cap Y_4$ then $u\in D_1$, contradicting the fact that $N(y)\cap D_1=\emptyset$. Thus $Y\cap Y_4=\emptyset$. Therefore, $X,Y, X_2, X_5, Y_2, Y_4,D_1, D_3$ are pairwise disjoint (as shown in Figure 6). Then
\[
n=|V(G)|\geq 8\cdot\left(\left\lfloor\frac{n}{8}\right\rfloor+1\right) >n,
\]
a contradiction.

 {\bf Case 2. } $xv_5,yv_2\notin E(G)$.

  \begin{figure}[H]
	\centering
	\begin{tabular}{@{\extracolsep{\fill}}c@{}c@{\extracolsep{\fill}}}
		\includegraphics[width=0.4\linewidth]{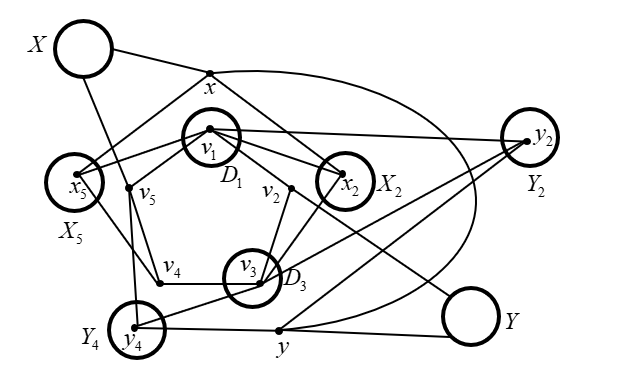}
	\end{tabular}
	\caption{$G$ in Case 2 of Claim \ref{claim-4} }
	\label{fig:image_with_table}
\end{figure}
Let $X=N(x)\cap N(v_{5})$ and $Y=N(y)\cap N(v_{2})$. Since $x\in W_1$, $X\cap D_1=\emptyset =X\cap D_3$. Since $y\in W_3$, $Y\cap D_1=\emptyset =Y\cap D_3$. Since $xy\in E(G)$, $X\cap Y=\emptyset$, $X\cap Y_j=\emptyset$ for $j=2,4$ and $X_i\cap Y=\emptyset$ for $i=2,5$. Thus $X\cap (Y\cup Y_2\cup Y_4\cup D_1\cup D_3)=\emptyset$ and $Y\cap (X\cup X_2\cup X_5\cup D_1\cup D_3)=\emptyset$.

Since $v_4v_5\in E(G)$, $X\cap X_5=\emptyset$.
If there exists $u\in X\cap X_2$ then $u\in D_4$, contradicting the fact that $N(x)\cap D_4=\emptyset$. Thus $X\cap X_2=\emptyset$. Since $v_1v_2\in E(G)$, $Y\cap Y_2=\emptyset$. If there exists $u\in Y\cap Y_4$ then $u\in D_1$, contradicting the fact that $N(y)\cap D_1=\emptyset$. Thus $Y\cap Y_4=\emptyset$. Therefore, $X,Y, X_2, X_5, Y_2, Y_4,D_1, D_3$ are pairwise disjoint (as shown in Figure 7). Then
\[
n=|V(G)|\geq 8\cdot\left(\left\lfloor\frac{n}{8}\right\rfloor+1\right) >n,
\]
a contradiction.

 {\bf Case 3. } $xv_2,yv_4\notin E(G)$.

  \begin{figure}[H]
	\centering
	\begin{tabular}{@{\extracolsep{\fill}}c@{}c@{\extracolsep{\fill}}}
		\includegraphics[width=0.4\linewidth]{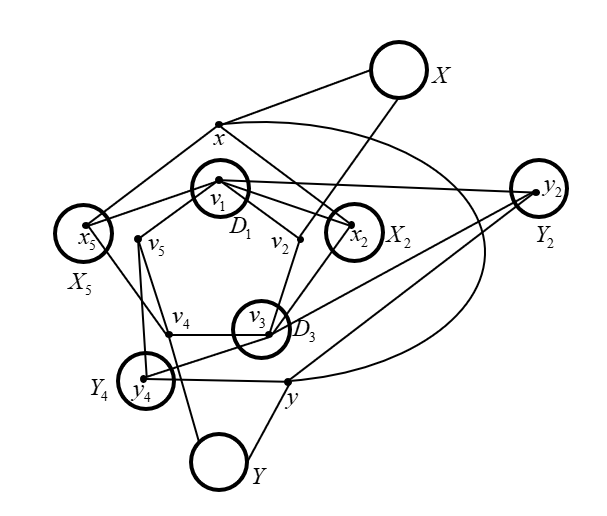}
	\end{tabular}
	\caption{$G$ in Case 3 of Claim \ref{claim-4} }
	\label{fig:image_with_table}
\end{figure}

Let $X=N(x)\cap N(v_{2})$ and $Y=N(y)\cap N(v_{4})$. Since $x\in W_1$, $X\cap D_1=\emptyset =X\cap D_3$. Since $y\in W_3$, $Y\cap D_1=\emptyset =Y\cap D_3$. Since $xy\in E(G)$, $X\cap Y=\emptyset$, $X\cap Y_j=\emptyset$ for $j=2,4$ and $X_i\cap Y=\emptyset$ for $i=2,5$. Thus $X\cap (Y\cup Y_2\cup Y_4\cup D_1\cup D_3)=\emptyset$ and $Y\cap (X\cup X_2\cup X_5\cup D_1\cup D_3)=\emptyset$.

Since $v_2v_3\in E(G)$, $X\cap X_2=\emptyset$.
If there exists $u\in X\cap X_5$ then $u\in D_3$, contradicting the fact that $N(x)\cap D_3=\emptyset$. Thus $X\cap X_5=\emptyset$. Since $v_4v_5\in E(G)$, $Y\cap Y_4=\emptyset$. If there exists $u\in Y\cap Y_2$ then $u\in D_5$, contradicting the fact that $N(y)\cap D_5=\emptyset$. Thus $Y\cap Y_2=\emptyset$. Therefore, $X,Y, X_2, X_5, Y_2, Y_4,D_1, D_3$ are pairwise disjoint (as shown in Figure 8). Then
\[
n=|V(G)|\geq 8\cdot\left(\left\lfloor\frac{n}{8}\right\rfloor+1\right) >n,
\]
a contradiction.

 {\bf Case 4. } $xv_5,yv_4\notin E(G)$.

  \begin{figure}[H]
	\centering
	\begin{tabular}{@{\extracolsep{\fill}}c@{}c@{\extracolsep{\fill}}}
		\includegraphics[width=0.4\linewidth]{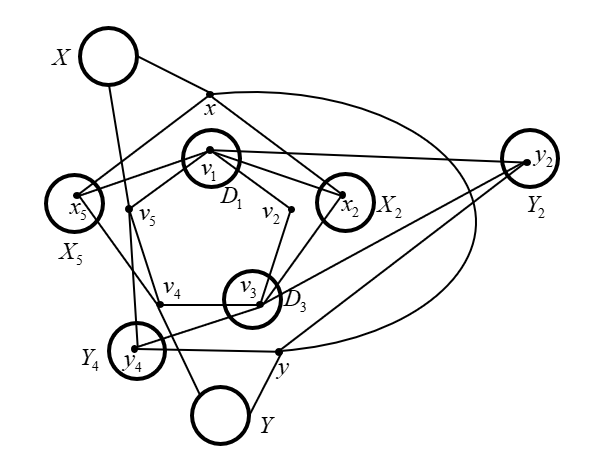}
	\end{tabular}
	\caption{$G$ in Case 4 of Claim \ref{claim-4} }
	\label{fig:image_with_table}
\end{figure}

Let $X=N(x)\cap N(v_{5})$ and $Y=N(y)\cap N(v_{4})$. Since $x\in W_1$, $X\cap D_1=\emptyset =X\cap D_3$. Since $y\in W_3$, $Y\cap D_1=\emptyset =Y\cap D_3$. Since $xy\in E(G)$, $X\cap Y=\emptyset$, $X\cap Y_j=\emptyset$ for $j=2,4$ and $X_i\cap Y=\emptyset$ for $i=2,5$. Thus $X\cap (Y\cup Y_2\cup Y_4\cup D_1\cup D_3)=\emptyset$ and $Y\cap (X\cup X_2\cup X_5\cup D_1\cup D_3)=\emptyset$.

Since $v_4v_5\in E(G)$, $X\cap X_5=\emptyset$.
If there exists $u\in X\cap X_2$ then $u\in D_4$, contradicting the fact that $N(x)\cap D_4=\emptyset$. Thus $X\cap X_2=\emptyset$. Since $v_4v_5\in E(G)$, $Y\cap Y_4=\emptyset$. If there exists $u\in Y\cap Y_2$ then $u\in D_5$, contradicting the fact that $N(y)\cap D_5=\emptyset$. Thus $Y\cap Y_2=\emptyset$. Therefore, $X,Y, X_2, X_5, Y_2, Y_4,D_1, D_3$ are pairwise disjoint (as shown in Figure 9). Then
\[
n=|V(G)|\geq 8\cdot\left(\left\lfloor\frac{n}{8}\right\rfloor+1\right) >n,
\]
a contradiction.
$\hfill\square$

Thus $G$ is a blow-up of $C_5$ with blocks $D_{1}\cup W_{1},D_{2}\cup W_{2},\ldots,D_{5}\cup W_{5}$.
$\hfill\square$

\begin{rem}
Note that we find 8 pairwise disjoint blocks of sizes at least $\delta_{2}(G)$ in several cases in the proof of Theorem \ref{main} (ii). However, only in Claim \ref{claim-2} the  8 blocks cover all the vertices of $G$.
Thus if $G$ is a triangle\text-free graph on $n$ vertices and $\delta_{2}(G)=\frac{n}{8}$ ($n$ is divisible by 8), then $G$ is a balanced M\"{o}bius ladder-blow-up.
\end{rem}

\noindent\emph{\textbf{Proof of Theorem \ref{C5}}.} Let $G$ be a $C_{5}$-free graph on $n$ vertices with $\delta_{2}(G)\geq 3$. By Lemma \ref{lem-2}, either $G$ is bipartite or contains a triangle. In the former case, there is nothing to prove. Thus we may assume that $G$ contains a triangle $v_1v_2v_3v_1$.

For any vertex $x\in V(G)\setminus \{v_1,v_2,v_3\}$, we claim $|N(x)\cap \{v_1,v_2,v_3\}|\geq 2$. Otherwise, by symmetry assume that $xv_1,xv_2\notin E(G)$. Then $\delta_2(G)\geq 3$ implies $|N(x)\cap N(v_{1})|\geq 3$ and $|N(x)\cap N(v_{2})|\geq 3$. It follows that there exist distinct vertices $y,z\in V(G)\setminus \{v_1,v_2,v_3,x\}$ such that $y\in N(x)\cap N(v_{1})$ and $z\in N(x)\cap N(v_{2})$. Then $xyv_1v_2zx$ forms a $C_5$, a contradiction. Thus $|N(x)\cap \{v_1,v_2,v_3\}|\geq 2$ for all $x\in V(G)\setminus \{v_1,v_2,v_3\}$.

Now we distinguish two cases.

\textbf{Case 1.} There exists $x\in V(G)\setminus \{v_1,v_2,v_3\}$ such that $|N(x)\cap \{v_1,v_2,v_3\}|=2$. Without loss of generality, assume  $xv_1,xv_2\in E(G)$. Then $\delta_2(G)\geq 3$ implies  $|N(x)\cap N(v_{3})|\geq 3$. Note that $v_{1},v_{2}$ are in $N(x)\cap N(v_{3})$. Since $|N(x)\cap N(v_{3})|\geq 3$, there exists $u\in V(G)\setminus\{v_1,v_2,v_3,x\}$ such that $u\in N(x)\cap N(v_{3})$. Therefore $v_{1}v_{2}v_{3}uxv_{1}$ is a $C_5$, a contradiction.

\textbf{Case 2.} $|N(x)\cap \{v_1,v_2,v_3\}|=3$ for all $x\in V(G)\setminus \{v_1,v_2,v_3\}$. Since $n\geq 5$, there exists two vertices $x,y\in V(G)\setminus \{v_1,v_2,v_3\}$ such that $x,y$ are both connected to $v_{1},v_{2},v_{3}$. Therefore $v_{1}xv_{2}yv_{3}v_{1}$ is a $C_5$, a contradiction.

Thus we conclude that $G$ is bipartite.
$\hfill\square$

\begin{rem}
If $G$ is a $C_{5}$-free graph on 3 vertices and $\delta_{2}(G)\geq 3$, then $G$ is a triangle. If $G$ is a $C_{5}$-free graph on 4 vertices and $\delta_{2}(G)\geq 3$, then $G$ is a $K_{4}$.
\end{rem}

\end{document}